\documentclass{amsart}
\usepackage{tikz-cd}
\usepackage{tikz}
\usetikzlibrary{arrows}
\usepackage{enumerate}

\def\Mbar{\overline{\mathcal{M}}}
\def\M{\mathfrak{M}}

\newtheorem{theorem}{Theorem}[section]
\newtheorem{lemma}[theorem]{Lemma}

\theoremstyle{definition}
\newtheorem{definition}[theorem]{Definition}
\newtheorem{example}[theorem]{Example}

\theoremstyle{remark}

\newtheorem{notation}[theorem]{Notation}

\numberwithin{equation}{section}

\begin{document}

\title{Tautological relations for stable maps to a target variety}

\author{Younghan Bae}
\address{Department of Mathematics, ETH Z\"{u}rich}
\curraddr{Department of Mathematics, ETH Z\"{u}rich, Z\"{u}rich 8092, Switzerland}
\email{younghan.bae@math.ethz.ch}
\thanks{The author was supported in part by ERC Grant ERC-2017-AdG-786580-MACI and Korea Foundation for Advanced Studies}



\date{March, 2020}


\begin{abstract}
We define tautological relations for the moduli space of stable maps to a target variety. Using the double
ramification cycle formula for target varieties of Janda-Pandharipande-Pixton-Zvonkine \cite{J2}, we construct non-trivial tautological relations parallel to Pixton's double ramification cycle relations for the moduli of curves. Examples and applications are discussed.
\end{abstract}

\maketitle



\section{Introduction}

\subsection{Tautological relations on the moduli space of stable curves}
Let $\overline{\mathcal{M}}_{g,n}$ be the moduli space space of stable curves of genus $g$ with $n$ marked points.
It is a smooth Deligne-Mumford stack of dimension $3g-3+n$ and has both singular and Chow cohomology theories.
Mumford \cite{M} initiated the study of the subring of tautological classes
\[
R^{*}(\overline{\mathcal{M}}_{g,n}) \subseteq A^{*}(\overline{\mathcal{M}}_{g,n})_{\mathbb{Q}}
\]
of the Chow ring, which we now call the {\em tautological ring}---see \cite[Section 1]{FP} for an introduction and basic definitions.
One advantage of considering the tautological ring is that there is a canonical set of additive generators of
$R^{*}(\overline{\mathcal{M}}_{g,n})$ by boundary strata together with basic $\psi$ and $\kappa$ classes \cite{G}.
A fundamental goal is to find all linear relations among these additive generators---such relations
are called \text{\em tautological relations}. For an excellent survey of this topic, see \cite{P1}. 

Because few tools exist to handle the geometry of moduli spaces of stable curves in higher genus,
finding new tautological relations is not easy. A breakthrough appeared in Pixton's note \cite{P11}, which was
inspired by the Faber-Zagier conjecture on tautological relations in $A^{*}(\mathcal{M}_{g})$ proven in
\cite{PP1}.
Pixton conjectured a systematic means of writing down tautological relations in terms of graph sums.
His conjecture was proven in cohomology \cite{PP} using the theory of Witten's 3-spin class and in Chow \cite{J1} using
the equivariant Gromov-Witten theory of $\mathbb{P}^{1}$.
 
\subsection{Tautological relations on the moduli space of stable maps} 
Let $X$ be a nonsingular projective variety over $\mathbb{C}$. For an effective curve class $\beta \in H_{2}(X, \mathbb{Z})$
and nonnegative integers $g$ and $n$,
we consider the moduli space of stable maps $\overline{\mathcal{M}}_{g,n,\beta}(X)$. A closed point of the moduli
space corresponds to a morphism from a connected nodal curve $C$ of genus $g$ with $n$ marked points to $X$,
\begin{equation}\label{129922}
f : (C, x_{1}, \cdots, x_{n}) \to X, \, \ \  f_{*}[C] = \beta \, .
\end{equation}
By the stability condition, the data \eqref{129922} is required to have only finitely many automorphisms---see  \cite{FP1}
for a foundational treatment.
Unlike $\overline{\mathcal{M}}_{g,n}$, the moduli space of stable maps is typically not smooth and
can have many connected components of different dimensions.
Nevertheless, $\overline{\mathcal{M}}_{g,n,\beta}(X)$ has a natural perfect obstruction theory of expected dimension 
\[
 \textup{vdim} = (1-g)(\text{dim}_{\mathbb{C}} X - 3) +\int_{\beta} c_{1}(X) +n\, ,
 \]
which allows us to define virtual fundamental classes \cite{BM} and Gromov-Witten invariants \cite{BF} for the target variety $X$ .

 The first question that we pursue here is whether a ring of tautological classes can be defined in the rational Chow theory
 of $\overline{\mathcal{M}}_{g,n,\beta}(X)$.
In Section 2, we define tautological classes via fundamental classes and show that their span in the Chow group carries a natural product structure. 
The issue of the multiplication of tautological classes is delicate because
we are considering cycles on a singular space $\overline{\mathcal{M}}_{g,n,\beta}(X)$.
Nevertheless, this product structure can be obtained by realizing tautolgical classes as Fulton's operational Chow classes \cite[Chapter 17]{F}.
In genus $0$, tautological classes were studied earlier by Oprea \cite{O} in cases
where the moduli space of stable maps $\overline{\mathcal{M}}_{0,n,\beta}(X)$ 
is smooth.
Our definition agrees with \cite{O} in the cases he considers.
  
Classical examples of tautological relations on $\overline{\mathcal{M}}_{g,n,\beta}(X)$  come from tautological relations on $\overline{\mathcal{M}}_{g,n}$ by pulling-back relations via the stabilization morphism 
 \begin{equation} \label{11}
 st : \overline{\mathcal{M}}_{g,n,\beta}(X) \to \overline{\mathcal{M}}_{g,n}\,.
\end{equation}
Formulas for the pull-back are presented in Section 2.2.

Our study is motivated by the following two basic questions concerning the structure of tautological relations on
the moduli space of stable maps:
  \begin{enumerate}
  \item Can we use the geometry of $X$ to find tautological relations on $\overline{\mathcal{M}}_{g,n,\beta}(X)$ that are not obtained via pull-back from Pixton's set of tautological relations on $\overline{\mathcal{M}}_{g,n}$?
 
 \item Can we use the tautological relations on $\overline{\mathcal{M}}_{g,n,\beta}(X)$ to determine new relations among Gromov-Witten invariants?
 
 \end{enumerate}
  
 For question (1), the answers should be in the form of general constructions which are valid for all target variety\footnote{For instance, if $X$ is a K3 surface, the virtual fundamental class of the moduli space of stable maps is zero if $\beta$ is nonzero. And it is not likely to be true that these relations come from tautological relations on $\overline{\mathcal{M}}_{g,n}$. In this note we will \textit{not} consider such cases.}. 
 As an example, in \cite{LP}, the authors obtained tautological relations in the Picard group of $\overline{\mathcal{M}}_{0,n, d}(\mathbb{P}^{m})$
 to prove a reconstruction theorem for genus 0 quantum cohomology and quantum K-theory for $$X\subseteq \mathbb{P}^m\, .$$
 Our main result here is a general construction of
 tautological relations for $\overline{\mathcal{M}}_{g,n,\beta}(X)$ 
using the new double
 ramification cycle formula for target varieties of Janda-Pandhari\-pande-Pixton-Zvonkine \cite{J2}.
 Because the construction essentially involves the geometry of $X$, the relations are not
 expected to be pull-backs. Examples and applications are provided in Section 4.

When $X$ is a point, the relations constructed here specialize to the double ramification cycle relations for $\overline{\mathcal{M}}_{g,n}$ conjectured by Pixton \cite{P12} and proven by Clader and Janda in \cite{CJ}. In fact, our proof follows the strategy of \cite{CJ}.

 \subsection*{Acknowledgements}   
  The author is indebted to Johannes Schmitt for his collaboration in proving Lemma 2.6, Longting Wu and Honglu Fan for their collaboration in computing Example 4.2, and Dragos Oprea for sharing his insights on tautological relations on the stable map spaces and pointing out an error in the first version of this paper. Special thanks to my Ph.D. advisor Rahul Pandharipande who suggested this problem to me. The author would also like to thank Felix Janda, Drew Johnson, Hyenho Lho, Georg Oberdieck, Aaron Pixton and Dimitri Zvonkine for helpful conversations. 
  
  This project has received funding from the European Research Council (ERC) under the European Unions Horizon 2020 research and innovation programme (grant agreement No. 786580) and Korea Foundation for Advanced Studies.

 \section{X-valued Tautological Ring}
\subsection{X-valued stable graphs}
 Let $X$ be a nonsingular projective variety over $\mathbb{C}$ and let $C(X)$ be the semigroup of effective curve classes in $H_{2}(X,  \mathbb{Z})$. For each element $\beta$ in $C(X)$, an Artin stack $\mathfrak{M}_{g,n,\beta}$ exists that parametrizes genus $g$, $n$ marked prestable curves $C$  together with a labeling on each irreducible component of $C$ by an element of $C(X)$. The labeling 
must satisfy the stability condition and the degrees must sum to $\beta$, see \cite[Section 2]{C} for details. This space is called the \textit{moduli space of prestable curves with $C(X)$-structure}. It is a smooth Artin stack 
and is  \'{ e}tale over the moduli space of marked prestable curves $\mathfrak{M}_{g,n}$.

Let us review the notion of $X$-valued stable graphs following \cite{J2}.
 
 \begin{definition}
 An \textit{X-valued stable graph} $\Gamma \in \mathcal{S}_{g, n, \beta}(X)$ consists of the data
 \[
 \Gamma = (V,\, H,\, L,\, g: V \to \mathbb{Z}_{\geq 0},\, v : H \to V,\, \iota : H \to H,\, \beta : V \to C(X) )
\]
satisfying the properties:
\begin{enumerate}[(i)]
    \item $V$ is a vertex set with a genus function $g : V \to \mathbb{Z}_{\geq 0}\,,$
    \item $H$ is a half-edge set equipped with a vertex assignment $v : H \to V$ and an involution $\iota$,
    \item $(V, H,\iota)$ defines a connected graph satisfying the genus condition
    \[
    \sum_{v \in V} g(v) + h^{1}(\Gamma) = g\,,
    \]
    \item for each vertex $v \in V$, the stability condition holds: if $\beta(v) = 0$, then
    \[
    2g(v) - 2 + n(v) > 0\,, 
    \] 
    where $n(v)$ is the valence of $\Gamma$ at $v$,
    \item the degree condition holds:
    \[
    \sum_{v \in V} \beta(v) = \beta\,.
    \]
\end{enumerate}
\end{definition}
\noindent
For a given graph $\Gamma$, denote $E$ as the set of 2-cycles of $\iota$ corresponding to edges, and $L$ as the set of fixed points of $\iota$ corresponding to the set of $n$ markings.
An automorphism of $\Gamma \in \mathcal{S}_{g, n, \beta}(X)$ consists of automorphisms of the sets $V$ and $H$, which leaves invariant the structures $L, g, v, \iota$, and $\beta$. Let Aut($\Gamma$) denote the automorphism group of $\Gamma$.

For each $X$-valued stable graph $\Gamma$, consider a moduli space $\overline{\mathcal{M}}_{\Gamma}$ of stable maps of the prescribed degeneration together with the canonical morphism \cite[Section 0.2]{J2}
\[
j_{\Gamma} : \overline{\mathcal{M}}_{\Gamma} \to \overline{\mathcal{M}}_{g,n,\beta}(X)\,.
\]
To construct $j_{\Gamma}$, a universal family of stable curves over $\overline{\mathcal{M}}_{\Gamma}$ is required \cite{BM}. The forgetful morphism from $\overline{\mathcal{M}}_{\Gamma}$ to $\mathfrak{M}_{\Gamma}$ carries a relative perfect obstruction 
theory.
The universal curve 
 \begin{equation}
 \pi : \mathcal{C}_{g,n,\beta}(X) \to \overline{\mathcal{M}}_{g,n,\beta}(X)
 \end{equation}
 has $n$ sections $s_{i}$ and a universal evaluation morphism
   \begin{equation}
 f : \mathcal{C}_{g,n,\beta}(X) \to X\,. 
 \end{equation}
Next, the notion of the strata algebra can be extended to $X$-valued stable graphs. 

 \begin{definition}
 A \textit{decorated $X$-valued stable graph $[\Gamma, \gamma]$} is an $X$-valued stable graph $\Gamma \in \mathcal{S}_{g,n,\beta}(X)$ together with the following data $\gamma$\,:
\begin{enumerate}[(i)]
    \item each leg $i \in L$ is decorated with $ev_{i}^{*}\alpha_{i}$, $\alpha_{i} \in A^{*}(X)\,,$
    \item each half-edge $h \in H$ is decorated with a monomial $\psi_{h}^{y[h]}$ for some $y[h] \in \mathbb{Z}_{\geq 0}\,,$
    \item each edge $e \in E$ is decorated with a monomial $ev_{e}^{*}\alpha_{e}$, $\alpha_{e} \in A^{*}(X)$,
    \item each vertex $v \in V$ is decorated with a product of {\em twisted $\kappa$ classes} 
\[
\kappa_{a_{1}, \cdots, a_{m}}(\alpha_{1}, \cdots, \alpha_{m}) = \pi_{m *}(\psi_{n+1}^{a_{1}+1} ev_{n+1}^{*} \alpha_{1} \cdots \psi_{n+m}^{a_{m}+1} ev_{n+m}^{*}\alpha_{n})
\]
where $\pi_{m}: \overline{\mathcal{M}}_{g,n+m,\beta}(X) \to \overline{\mathcal{M}}_{g,n,\beta}(X)$ is the morphism forgetting the last $m$ marked points, $\alpha_{1}, \ldots, \alpha_{m} \in A^{*}(X)$ and $a_{1}, \ldots, a_{m}\in \mathbb{Z}_{\geq -1}$.
\end{enumerate}
\end{definition}
\noindent
We bound the degree of $\gamma$ by the virtual dimension of the moduli space
\[
\text{deg}(\gamma) \leq \textup{vdim} \, \overline{\mathcal{M}}_{g, n, \beta}(X)\,.
\]
Consider the $\mathbb{Q}$-vector space $\mathcal{S}_{g,n,\beta}(X)$ whose basis is given by the isomorphism classes of decorated $X$-valued stable graph $[\Gamma, \gamma ]$. We give a product structure on $\mathcal{S}_{g,n,\beta}(X)$ as follows, see \cite{G, J2}. Let
\[
 [\Gamma_{A}, \gamma_{A}], [\Gamma_{B}, \gamma_{B}] \in \mathcal{S}_{g,n,\beta}(X)
\]
be two elements in $\mathcal{S}_{g,n,\beta}(X)$. The product of two elements is a finite linear combination of decorated $X$-valued graphs as follows:

\begin{enumerate}
\item Consider $\Gamma \in \mathcal{S}_{g,n,\beta}(X)$ with edges colored by $A$ or $B$ so that after contracting the edge not colored as $A$ (resp. $B$), we obtain $\Gamma_{A}$ (resp. $\Gamma_{B}$).

\item For each such $\Gamma$, we assign operational Chow classes by the following rules.

$\bullet$ The operational Chow classes on legs are obtained by multiplying the corresponding leg monomials on $\gamma_{A}$ and $\gamma_{B}$.

$\bullet$ The operational Chow classes on a half edge colored $A$ only or colored $B$ only is determined by $\gamma_{A}$ or $\gamma_{B}$ respectively. On an edge $e=(h, h')$ colored both $A$ and $B$, we multiply by
 \begin{equation} \label{2211}
- (\psi_{h} + \psi_{h'})
\end{equation}
in addition to contributions from $\gamma_{A}$ and $\gamma_{B}$.

$\bullet$ The factors of edges coming from $\gamma_{A}$ and $\gamma_{B}$ descend to $\Gamma$.

$\bullet$ The factors on the vertex $v$ in $\Gamma_{A}$ (resp. $\Gamma_{B}$) are split in all possible ways among the vertices which collapse to $v$ as $\Gamma$ is contracted to $\Gamma_{A}$ and $\Gamma_{B}$. Then we multiply two vertex contributions.
\end{enumerate}
The above product yields a $\mathbb{Q}$-algebra structure on $\mathcal{S}_{g,n,\beta}(X)$.
Push-forward along $j_{\Gamma}$ defines a $\mathbb{Q}$-linear map $q : \mathcal{S}_{g,n,\beta}(X) \to A_{*}(\overline{\mathcal{M}}_{g,n,\beta}(X))$,
 \begin{equation}
 q([\Gamma, \gamma]) = j_{\Gamma *} \left( \gamma \cap [\overline{\mathcal{M}}_{\Gamma}]^{vir} \right) \in A_{*}(\overline{\mathcal{M}}_{g,n,\beta}(X))\,.
\end{equation}
An element of the kernel of $q$ is called a \textit{tautological relation for the target variety $X$}. Unlike in \cite{PP}, there is no intersection product on the Chow group of $A_{*}(\overline{\mathcal{M}}_{g,n,\beta}(X))$, so there is no obvious reason why the kernel of the map $q$ is an ideal of $\mathcal{S}_{g,n,\beta}(X)$. However, the map $q$ factors through the operational Chow ring of $\overline{\mathcal{M}}_{g,n,\beta}(X)$.
\begin{definition}
Let $\mathfrak{M}$ be an Artin stack over $\mathbb{C}$. For each scheme $U$ and morphism $U \to \mathfrak{M}$, an operational Chow class $\alpha \in A^{p}(\mathfrak{M})$ is a collection of morphisms
\[
\alpha_{U} : A_{*}(U) \to A_{*-p}(U)
\]
which is compatible with proper push-forward, flat pull-back, locally complete intersection(l.c.i) pull-back, and Chern classes, see \cite[Chapter 17.1]{F}.
\end{definition}
For example, Chern classes of a vector bundle on a scheme are operational Chow classes. Because the pullback morphism is defined for operational Chow groups, operational Chow classes have well-defined product structure.
\begin{lemma} 
The kernel of $q$ is an ideal in $\mathcal{S}_{g,n,\beta}(X)$.
\end{lemma}
\begin{proof} 
From the splitting axiom, the map $q$ factors through the following diagram
\begin{equation}
    \begin{tikzcd}
      \mathcal{S}_{g,n,\beta}(X) \arrow{r}{q_{0}} \arrow[swap]{dr}{q} &A^{*}(\overline{\mathcal{M}}_{g,n,\beta}(X)) \arrow{d}{\cap [\overline{\mathcal{M}}_{g,n,\beta}(X)]^{vir}} \\
      & A_{*}(\overline{\mathcal{M}}_{g,n,\beta}(X))\,.
    \end{tikzcd}
\end{equation}
The map $q_{0}$ assigns $[\Gamma, \gamma]$ to an operational Chow class on $\overline{\mathcal{M}}_{g,n,\beta}(X)$ as follows. For each scheme $U$ and a morphism $u : U \to \overline{\mathcal{M}}_{g,n,\beta}(X)$, consider the following fiber diagram
\begin{equation*}
    \begin{tikzcd}
    U' \arrow[r,"j_{u}"] \arrow[d, "u' "]
    & U \arrow[d, "u"] \\
    \overline{\mathcal{M}}_{\Gamma} \arrow[r, "j_{\Gamma}"] \arrow[d]
    & \overline{\mathcal{M}}_{g,n,\beta}(X) \arrow[d,] \\
   \mathfrak{M}_{\Gamma} \arrow[r,"\xi_{\Gamma}" ]
   &\mathfrak{M}_{g,n,\beta}
    \end{tikzcd}
\end{equation*}
where $\mathfrak{M}_{\Gamma}$ is the moduli space of prestable curves with prescribed degeneration defined by $\prod_{v \in V(\Gamma)} \mathfrak{M}_{g(v),n(v),\beta(v)}$ and $\xi_{\Gamma}$ is the gluing morphism. The class 
\begin{equation*}
    q_{0}([\Gamma, \gamma])(s) : A_{*}(U) \to A_{* - |E| - \textup{deg}(\gamma)}(U)
\end{equation*} 
assigns $V \in A_{*}(U)$ to 
\begin{equation*}
    j_{s *}(u'^{*}(\gamma) \cap \xi_{\Gamma}^{!}V)\,.
\end{equation*}
A parallel argument to that given  in \cite[Appendix]{G} shows that the map $q_{0}$ is a ring homomorphism. Therefore the kernel of $q$ is an ideal in $\mathcal{S}_{g,n,\beta}(X)$.
\end{proof}
\noindent From the above, the image of the map $q$ has a $\mathbb{Q}$-algebra structure. We call this ring as the \textit{tautological ring of the moduli space of stable maps to $X$} and write $R^{*}(\overline{\mathcal{M}}_{g,n,\beta}(X))$.


\subsection{Tautological relations from Pixton's 3-spin relations}

We briefly review tautological relations on $\overline{\mathcal{M}}_{g,n,\beta}(X)$ coming from tautological relations on the moduli space of stable curves via the stabilization morphism \eqref{11}. We assume $2g-2+n > 0$ throughout this section.
The following lemma illustrates the pull-back formula under the stabilization morphism.
\begin{lemma}\footnote{This lemma was formulated with Johannes Schmitt.}
 Let $st : \overline{\mathcal{M}}_{g,n,\beta}(X) \to \overline{\mathcal{M}}_{g,n}$ be the stabilization morphism. The following hold:
\begin{enumerate}
\item Let $[\Gamma]$ be a boundary class in $R^{*}(\overline{\mathcal{M}}_{g,n})$. Then
\[
st^{*}\frac{1}{|\textup{Aut}(\Gamma)|}[\Gamma] = \sum_{\Gamma'} \frac{1}{|\textup{Aut}(\Gamma')|} [\Gamma']\,.
\]
The sum is over all $X$-valued stable graphs $\Gamma'$ where the graph of $\Gamma'$ is equal to $\Gamma$ but the degree map $\beta(v)$ can vary. 

\item $st^{*}\psi_{i} = \psi_{i} - [D_i]$, where $D_{i}$ is an $X$-valued stable graph with one edge connecting two vertices $v_{1}$ and $v_{2}$ with $g(v_{1}) = g$ so that the only leg attached to $v_{2}$ is the $i$-th leg.  
\item The pull-back of $\kappa$ classes are tautological.
\end{enumerate}
\end{lemma}

\begin{proof} It is sufficient to prove the lemma for the stabilization morphism $\mathfrak{M}_{g,n,\beta} \to \overline{\mathcal{M}}_{g,n} $. Proof of (1) and (2) are well known, see \cite{BM}.

Proof of (3). The pull-back formula follows from the fact that the log canonical sheaf $\omega_{\textup{log}}$ does not change under the stabilization for semistable curves. Consider the following diagram
\begin{equation}
\begin{tikzcd}
  \mathfrak{M}_{g,n+1,\beta} \arrow[r, "st"] \arrow[rd, "\pi_{1}"] & \tilde{\mathfrak{C}} \arrow[d, "\pi_{2}"]\\
                                          & \mathfrak{M}_{g,n,\beta}
       \end{tikzcd}
\end{equation}
where $\tilde{\mathfrak{C}}$ is the pull-back of the universal curve over $\overline{\mathcal{M}}_{g,n}$ under the stabilization morphism and $st : \mathfrak{M}_{g,n+1,\beta} \to \tilde{\mathfrak{C}}$ is the fiberwise stabilization. Let $\omega_{i}=\omega_{\pi_{i}}^{\textup{log}}$, $i=1, 2$, be the log canonical sheaves for each projections. Then, 
\begin{equation*}
    st^{*}c_{1}(\omega_{2}) = c_{1}(\omega_{1}) - [D]\,,
\end{equation*}
 where D is the sum of stable graphs with one edge connecting two vertex $v_{1}$ and $v_{2}$ where $g(v_{1})=g, g(v_{2})=0$ and the only leg adjacent to $v_{2}$ is the $n+1$-th leg. Because the morphism $st$ is birational, 
\begin{equation*}
    st^{*}\kappa_{n} = \pi_{1 *} (c_{1}(\omega_{1}) - D)^{n+1}\,.
\end{equation*}
After expanding the right hand side, we get the pull-back formula.
\end{proof}

\begin{example}
 From the above lemma, we get
 \[
 st^{*}\kappa_{1} = \kappa_{1} + [D]
 \]
  in $R^{1}(\overline{\mathcal{M}}_{g,n,\beta}(X))$ where $D$ is the sum of $X$-valued stable graph with one edge connecting two vertices $v_{1}$ and $v_{2}$ where $g(v_{1}) = g, g(v_{2}) = 0$ and no leg attached to $v_{2}$.
\end{example}
We can also define the notion of tautological relations on the stack of prestable curves $\M_{g,n}$ \cite{BS}. It is natural to ask whether there are more relations on $\M_{g,n}$ because the stabilization morphism factors through $\M_{g,n}$. Unlike $\Mbar_{g,n}$, the stack $\M_{g,n}$ involves difficulties to find relations. A full description of tautological relations on $\mathfrak{M}_{0,n}$ will appear in \cite{BS}.


\section{Tautological relations from double ramification cycles}

\subsection{Twisted double ramification relations}
 Let $S$ be a line bundle over $X$ and let $k \in \mathbb{Z}$ be an integer. A vector $A=(a_{1}, \cdots, a_{n}) \in \mathbb{Z}^{n}$ of \textit{k-twisted double ramification data for a target variety} is defined by the condition
 \begin{equation} \label{31}
 \sum_{i=1}^{n} a_{i} = \int_{\beta} c_{1}(S) + k(2g-2+n)\,.
\end{equation}
Under this condition, we define $k$-twisted double ramification relations as follows. For each genus $g$, $n$ pointed nodal curve $(C, x_{1}, \cdots, x_{n})$, let 
\[
\omega_{\textup{log}} = \omega_{C}(x_{1}+ \cdots+ x_{n})
\]
be the log canonical line bundle. In (2.1) and (2.2), we defined the universal curve and the universal evaluation map
\begin{equation*}
    \begin{tikzcd}
      \mathcal{C}_{g,n,\beta}(X) \arrow[r, "f"] \arrow[d, "\pi"] & X \\
      \overline{\mathcal{M}}_{g,n,\beta}(X)\,.
    \end{tikzcd}
\end{equation*}
\noindent
The following tautological classes are obtained from the above diagram

$\bullet$ $\xi_{i}=c_{1}(s^{*}_{i}f^{*}S)\,,$

$\bullet$ $\xi = f^{*}c_{1}(S)\,,$ 

$\bullet$ $\eta_{a,b}=\pi_{*}(c_{1}((\omega_{\textup{log}}))^{a}\xi^{b})\,,$

$\bullet$ $\eta = \eta_{0,2} = \pi_{*}(\xi^{2})\,.$

\noindent
The subscript $i$ corresponds to the $i$-th marked point and $a, b$ are nonnegative integers. To state the double ramification (DR) vanishing formula, we recall the definition of weights for $X$-valued stable graphs.
\begin{definition}
   Let $\Gamma \in \mathcal{S}_{g,n,\beta}(X)$ be an $X$-valued stable graph. A \textit{k-weighting mod $r$} of $\Gamma$ is a function on the set of half edges, 
\[
w : H(\Gamma) \to \{0, 1, \cdots, r-1\}
\]
which satisfies:
\begin{enumerate}[(i)]
\item $\forall i \in L(\Gamma)$, corresponding to $i$-th marking, 
\[
w(i) = a_{i} \mod r\,,
\] 
\item $\forall e \in E(\Gamma)$, corresponding to half edges $h, h'$, 
\[
w(h) + w(h') = 0 \mod r\,,
\]
\item $\forall v \in V(\Gamma)$,
\[
\sum_{v(h)=v} w(h) = \int_{\beta(v)} c_{1}(S) + k(2g(v) - 2 + n(v)) \mod r\,,
\]
where the sum is over all half-edges incident to $v$.
\end{enumerate}
\end{definition}
\noindent We denote $\mathsf{W}_{\Gamma, r, k}$ be the set of all $k$-weightings mod $r$ of $\Gamma$. 
 \vspace{5pt}

 Let $A$ be a vector of $k$-twisted ramification data for genus $g$. For each positive integer $r$, let $\mathsf{P}^{d, r}_{g,A,\beta, k}$ be the degree $d$ component of the class
 \begin{equation}
 \begin{aligned} \label{311}
 {} & \sum_{\Gamma \in \mathcal{S}_{g,n,\beta}(X)} \sum_{w \in \mathsf{W}_{\Gamma,r,k}} \frac{r^{-h^1(\Gamma)}}{|\textup{Aut}(\Gamma)|} j_{\Gamma *} \Bigg[\prod_{i=1}^{n} \textup{exp}(\frac{1}{2}a_{i}^{2}\psi_{i} + a_{i}\xi_{i}) \\
&  \times \prod_{v} \textup{exp}(-\frac{1}{2}\eta(v) - k \eta_{1,1}(v) - \frac{1}{2} k^{2} \kappa_{1}(v)) 
  \prod_{e=(h, h')} \frac{1- \textup{exp} \big(-\frac{w(h)w(h')}{2}(\psi_{h} + \psi_{h'})\big)}{\psi_{h} + \psi_{h'}}\Bigg].
\end{aligned}
\end{equation}
The argument of \cite[Appendix]{J} can be directly applied to $X$-valued graph sums and the class $\mathsf{P}^{d, r}_{g,n,\beta, k}$ is polynomial in $r$ for sufficiently large $r$ \cite[Proposition 1]{J2}. Denote $\mathsf{P}^{d}_{g,n,\beta, k}$ to be the constant term. In the next section, we will prove $\mathsf{P}^{d}_{g,n,\beta, k}$ vanishes if $d>g$. 
 \subsection{The vanishing result} 
In this section, we closely follow Clader-Janda \cite{CJ} and Janda-Pandharipande-Pixton-Zvonkine \cite{J2}. We first outline constructions in \cite{J2}. For each positive integer $r$, we consider the moduli space $\overline{\mathcal{M}}^{r}_{g,A,\beta}(X, S)$ of stable maps $(f: (C, x_{1}, \cdots, x_{n}) \to X)$ from twisted prestable curve\footnote{Let $\mu_{r}$ be the group of $r$-th root of unity. On the \'etale neighborhood of a node, the orbifold structure of the family of twisted prestable curves 
\[ (x, y) \mapsto z\,, \,\, z = xy
\]
is given by taking $\mu_{r} \times \mu_{r}$ quotient in the domain and $\mu_{r}$ quotient in the target. For the definition of the moduli space of prestable curves, see \cite[Section 1.2]{J2}.
} $C$ with an $r$-th root of the line bundle 
 \[
 f^{*}S \otimes \omega^{\otimes k}_{\textup{log}}(- \sum_{i=1}^{n} a_{i}x_{i}).
 \]
Let $\mathfrak{M}^{r,L}_{g,n}$ be the stack of prestable twisted curves with a degree 0 line bundle without conditions on the stabilizers and let $\mathfrak{M}^{Z}_{g,n}$ be the stack of prestable curves with a degree 0 line bundle. We have a morphism
 \begin{equation}
 \overline{\mathcal{M}}_{g,n,\beta}(X) \to \mathfrak{M}^{Z}_{g,n}, \, \, [f] \to (C, f^{*}S \otimes \omega^{\otimes k}_{\textup{log}}(- \sum_{i=1}^{n} a_{i}x_{i}))
 \end{equation}
and a fiber diagram
\begin{equation} 
 \begin{tikzcd} 
 \overline{\mathcal{M}}^{r}_{g,A,\beta}(X, S) \arrow[r, "\epsilon"] \arrow[d]
    & \overline{\mathcal{M}}_{g,n,\beta}(X) \arrow[d,] \\
  \mathfrak{M}^{r, L}_{g,n} \arrow[r,"\epsilon" ]
&\mathfrak{M}^{Z}_{g,n}.
\end{tikzcd}
\end{equation}
Here, $\epsilon: \mathfrak{M}^{r, L}_{g,n} \to \mathfrak{M}^{Z}_{g,n}$ is the composition of two morphisms
\begin{equation*}
\M^{r,L}_{g,n} \to \M^{r,Z,\textup{triv}}_{g,n} \to \M^{Z}_{g,n}
\end{equation*}
where the first morphism maps $(C,L)$ to $(C,Z=L^{\otimes r})$, and the second morphism comes from the $r$-th root construction on boundary divisors.
\begin{theorem}
If $d>g$, and $n \geq 1$, then $\mathsf{P}^{d}_{g,A,\beta,k} = 0 \in A_{\textup{vdim} -d} (\overline{\mathcal{M}}_{g,n,\beta}(X))$.
\end{theorem}
\begin{proof} 
We follow four basic steps of the vanishing argument of \cite{CJ}.

 \textit{Step 1. Reduction step.} 
 From the polynomiality of $\mathsf{P}^{d}_{g,A,\beta,k}$ in $A$, it is enough to prove the vanishing of \eqref{311} when exactly one $a_{i}$ is negative, see \cite[Section 5.3]{CJ}. Assume that $a_{1} < 0$ and $a_{i} \geq 0$ for all $i \geq 2$. Choose a positive number $r > \textup{max} \{ |a_{i}| \}$ and take $A' = (a'_{1},\cdots,a'_{n}) = (a_{1}+r, a_{2},\cdots,a_{n})$. The constraint \eqref{31} is still valid modulo $r$ and the constant term of \eqref{311} is invariant under the translation. Let $\mathfrak{M}_{g,n}^{Z'}$ be the stack of prestable curves together with a degree $-r$ line bundle, and $\mathfrak{M}^{r, L'}_{g,n}$ be the stack of prestable twisted curves together with a degree $-1$ line bundle. There is a universal twisted curve 
\begin{equation}\label{35}
\pi : \mathfrak{C}^{r, L'}_{g,n} \to \mathfrak{M}^{r, L'}_{g,n}
\end{equation}
and the universal line bundle $\mathcal{L}_{A'} \to \mathfrak{C}_{g,n}^{r, L'}$.

 \textit{Step 2. Use the polynomiality of $\mathsf{P}^{d}_{g,A,\beta,k}$.} From the polynomiality of $\mathsf{P}^{d}_{g,A',\beta,k}$ in $k$ \cite{PixtonZagier}, we may assume that $k$ is a negative number. Since $X$ is a projective variety, $S$ can be written as a difference of two ample line bundles $S_{1}$ and $S_{2}$. Introduce new variables $x, y$ and write $S = S_{1}^{\otimes x} \otimes S_{2}^{\otimes y}$. We can view the expression $\mathsf{P}^{d}_{g,A',\beta,k}(S)$ as a polynomial in $x, y, a'_{1}, \cdots, a'_{n-1}$, after substituting \eqref{31}. The polynomiality ensures that it suffices to prove vanishing for $x, y \ll 0$.
 
 \textit{Step 3. Use stability and the degree analysis of \cite[Lemma 4.2]{CJ}.} 
Consider the morphism \eqref{35}.  The degree of the orbifold line bundle $\mathcal{L}_{A'}$ restricted to each fiber on the irreducible component $C_{v}$ is
 \[
\frac{1}{r} \Bigg[ \int_{\beta_{v}} c_{1}(S) + k(2g_{v} - 2 + n_{v}) - \sum_{i \dashv v} a'_{i} \Bigg]\,.
 \] 
 If $\beta_{v}$ is $0$, $2g_{v} - 2 + n_{v} > 0$ and the degree of the line bundle is negative. If $\beta_{v}$ is nonzero, $2g_{v} - 2 + n_{v}$ can be negative. However, we assumed that $x,y \ll 0$ so the degree of the line bundle is also negative. The proof of \cite[Lemma 4.2]{CJ} implies $-R\pi_{*}\mathcal{L}_{A'} $ is a locally free sheaf of rank $g$. 

\textit{Step 4. Use the Grothendieck-Riemann-Roch formula in \cite[Section 2]{J2}.}
From the computation in \cite[Corollary 11]{J2}, $r^{-2g+2d+1}\epsilon_{*}c_{d}(-R\pi_{*}\mathcal{L}_{A'})$ is obtained by substituting $r = 0$ into the degree $d$ part of
\begin{equation}
\begin{split} \label{36}
{} & \sum_{\Gamma \in \mathcal{S}_{g,n,\beta}(X)} \sum_{w \in \mathsf{W}_{\Gamma,r,k}} \frac{r^{-h^{1}(\Gamma)}}{|\textup{Aut}(\Gamma)|} j_{\Gamma *} \Bigg[\prod_{v} \textup{exp}(-\frac{1}{2}\eta(v) - k \eta_{1,1}(v) - \frac{1}{2} k^{2} \kappa_{1}(v)) \\
& \prod_{e=(h, h')} \frac{1- \textup{exp} \big(-\frac{w(h)w(h')}{2}(\psi_{h} + \psi_{h'})\big)}{\psi_{h} + \psi_{h'}}\Bigg]\,.
\end{split}
\end{equation}
Twisting by $\omega_{\textup{log}}^{\otimes k}$ produces additional terms in \eqref{36}. The above formula vanishes for all $d$ greater than $g$. By the pull-back formula in \cite[Lemma 4]{J2}, we obtain the vanishing result on $\overline{\mathcal{M}}_{g,n,\beta}(X)$.
\end{proof}
In \cite{J2}, the untwisted double ramification cycle formula for the target variety was proven via relative/orbifold Gromov-Witten theory (the cycle $\mathsf{DR}_{g,n,\beta}(X,S)$ was defined in \cite[Definition 1]{J2}). In fact, a proof of the untwisted double ramification cycle relations was implicitly stated in \cite{J2}.
\begin{proof}[Proof for untwisted classes] An alternative proof follows from the localization formula in \cite[Section 3]{J2}. For $l \geq 1$, consider the class
\begin{equation} \label{34}
    \textup{Coeff}_{t^{0}}[ \epsilon_{*} ( t^{l} [\overline{\mathcal{M}}_{g,A,\beta}(\mathbb{P}(X,S)[r], D_{\infty})]^{vir}) ]
\end{equation}
in $A_{*}(\overline{\mathcal{M}}_{g,n,\beta}(X))$. From \cite[Proposition 9]{J2} and the localization formula, (\ref{34}) is a Laurent polynomial in $r$. The coefficient of $r^{1-l}$ of (\ref{34}), should vanish for all $l \geq 1$. If $l = 1$, we prove Theorem 3.3 and if $l \geq 2$, we get the vanishing result $\mathsf{P}^{g+l-1}_{g,A,\beta} = 0$.
\end{proof}
\section{Examples}
\noindent
From the double ramification cycle relations, tautological relations on the moduli space of stable maps to a target $X$ were constructed. Returning back to the questions in the introduction, it is natural to ask whether these relations can be obtained from tautological relations on $\mathfrak{M}_{g,n}$. Consider a system of ideals
\begin{equation}
    I_{g,n,\beta} \subset \mathcal{S}_{g,n,\beta}(X)\,,
\end{equation}
which is the smallest system satisfying three conditions:
\begin{enumerate}[(i)]
    \item $I_{g,n,\beta}$ contains pull-back of every tautological relations on $\mathfrak{M}_{g,n}$\footnote{Tautological relations on $\mathfrak{M}_{g,n}$ can be defined similarly as tautological relations on $\overline{\mathcal{M}}_{g,n}$. The precise definitions and computations  will be studied in \cite{BS}.},
    \item $I_{g,n,\beta}$ is closed under the map $\mathcal{S}_{g,n+1,\beta}(X) \to \mathcal{S}_{g,n,\beta}(X)$ induced by forgetting the last marked point,
    \item The system of ideals $\{I_{g,n,\beta} \}$ is closed under the map
    \begin{equation*}
        \prod_{v \in V(\Gamma)} \mathcal{S}_{g_{v},n_{v},\beta_{v}}(X) \to \mathcal{S}_{g,n,\beta}(X)
    \end{equation*} 
induced by any $X$-valued stable graph $\Gamma$. 
\end{enumerate}
We say a class $R \in \mathcal{S}_{g,n,\beta}(X)$ \textit{is obtained from moduli spaces of curves} if $R$ is a class in $I_{g,n,\beta}$.

In many cases, it is difficult to prove that a tautological relation in $\overline{\mathcal{M}}_{g,n,\beta}(X)$ does not come from moduli spaces of curves. On the other hand, we will see that some DR relations are obtained from moduli spaces of curves in nontrivial ways.


\subsection{Genus 0 case}\footnote{The following genus $0$ computation was done with Honglu Fan and Longting Wu.}
Let $S \to X$ be a line bundle over a nonsingular projective variety $X$. For each integer $d$ greater than $g$, we have the vanishing result $\mathsf{P}^{d}_{g,A,\beta} = 0$ proven in the previous section. In this section, we extract tautological relations for the space $X$ by using the polynomiality of  $\mathsf{P}^{d}_{g,A,\beta}$. The class $\mathsf{P}^{d}_{g,A,\beta}$ is polynomial in the ramification data $A$. So, each coefficients of the polynomial in variables in $A$ should vanish. 

We can further separate each relations by considering the degree of the line bundle. The degree constraint
\begin{equation*} \label{41}
\int_{\beta} c_{1}(S) - \sum_{i=1}^{n} a_{i} = 0
\end{equation*} 
has a scale invariance
 \begin{equation}
  a_{i} \mapsto m a_{i}\,, \, \, \, S \mapsto S^{\otimes m}
\end{equation}
for $m \in \mathbb{Z}$. Considering a family of relations $\mathsf{P}^{d}_{g,mA,\beta} = 0$ for all integer $m$, each basic tautological class has a degree with respect to $m$. Therefore each relation breaks into smaller relations according to the degree $m$.
\begin{notation} In the following examples, the symbol
\[
\begin{tikzpicture}[baseline=3pt,->,>=,node distance=1.3cm,thick,main node/.style={circle,draw,font=\Large,scale=0.5}]
  \node at (0,0.2) [main node] (A){g,$\beta$};
\end{tikzpicture} 
\]
denotes a vertex with genus $g$ with degree $\beta$. If the genus of the vertex is $0$, we omit $0$. We write 
\[
\Bigg[ \begin{tikzpicture}[baseline=3pt,->,>=,node distance=1.3cm,thick,main node/.style={circle,draw,font=\Large,scale=0.4}]
  \node at (0,0.2) [main node] (A){$\beta_{1}$};
  \node at (1,0.2) [main node] (B){$\beta_{2}$};
  \node at (-0.2,0.8) (p1) {\textup{1}};
  \node at (1.2,0.8) (p2) {\textup{2}};
  \draw [-] (A) to (p1);
  \draw [-] (A) to (B);
  \draw [-] (B) to (p2); 
  \end{tikzpicture} \Bigg]
\]
to indicate the sum over all possible stable splittings $\beta_{1}+\beta_{2}=\beta$. Also, we use the
notation $b = \int_{\beta} c_{1}(S)$ and $b_{i} = \int_{\beta(v_{i})} c_{1}(S)$.
\end{notation}

\begin{example}
Let $g = 0$, $n = 2$, $d = 1$. We show that all relations can be obtained from the moduli space of prestable curves. There are three prestable graphs with at most one edge. The term $\mathsf{P}^{1}_{0, A, \beta}$ is equal to 
\begin{align*}
(-\frac{1}{2}\eta + \frac{1}{2}a_{1}^{2}\psi_{1} + \frac{1}{2} a_{2}^{2}\psi_{2} + a_{1}\xi_{1} + a_{2}\xi_{2}) \Bigg[ \begin{tikzpicture}[baseline=3pt,->,>=,node distance=1.3cm,thick,main node/.style={circle,draw,font=\Large,scale=0.4}]
  \node at (0,0.2) [main node] (A){$\beta$};
  \node at (-0.2,0.8) (p1) {\textup{1}};
  \node at (0.2,0.8) (p2) {\textup{2}};
  \draw [-] (A) to (p1);
  \draw [-] (A) to (p2); 
  \end{tikzpicture} \Bigg]
- 
\frac{b_{2}^{2}}{2} \Bigg[ \begin{tikzpicture}[baseline=3pt,->,>=,node distance=1.3cm,thick,main node/.style={circle,draw,font=\Large,scale=0.4}]
  \node at (0,0.2) [main node] (A){$\beta_{1}$};
  \node at (1,0.2) [main node] (B){$\beta_{2}$};
  \node at (-0.2,0.8) (p1) {\textup{1}};
  \node at (0.2,0.8) (p2) {\textup{2}};
  \draw [-] (A) to (p1);
  \draw [-] (A) to (B);
  \draw [-] (A) to (p2); 
  \end{tikzpicture} \Bigg]  
-
\frac{(b_{1} - a_{1})^{2}}{2} \Bigg[ \begin{tikzpicture}[baseline=3pt,->,>=,node distance=1.3cm,thick,main node/.style={circle,draw,font=\Large,scale=0.4}]
  \node at (0,0.2) [main node] (A){$\beta_{1}$};
  \node at (1,0.2) [main node] (B){$\beta_{2}$};
  \node at (-0.2,0.8) (p1) {\textup{1}};
  \node at (1.2,0.8) (p2) {\textup{2}};
  \draw [-] (A) to (p1);
  \draw [-] (A) to (B);
  \draw [-] (B) to (p2); 
  \end{tikzpicture} \Bigg]
\end{align*}
After substituting $a_{2} = b - a_{1}$ in $\mathsf{P}^{1}_{0, A, \beta} = 0$, the following vanishing holds: 

$\bullet$ Coefficient of $a_{1}^{2}$: 
\begin{align*}
\psi_{1}
+
\psi_{2}
-
\Bigg[ \begin{tikzpicture}[baseline=3pt,->,>=,node distance=1.3cm,thick,main node/.style={circle,draw,font=\Large,scale=0.4}]
  \node at (0,0.2) [main node] (A){$\beta_{1}$};
  \node at (1,0.2) [main node] (B){$\beta_{2}$};
  \node at (-0.2,0.8) (p1) {1};
  \node at (1.2,0.8) (p2) {2};
  \draw [-] (A) to (p1);
  \draw [-] (A) to (B);
  \draw [-] (B) to (p2); 
  \end{tikzpicture} \Bigg]
  = 0\,,
\end{align*}

$\bullet$ Coefficient of $a_{1}$: 
\begin{align*}
\xi_{1}
-
\xi_{2}
-
b\psi_{2}
+
b_{1}\Bigg[ \begin{tikzpicture}[baseline=3pt,->,>=,node distance=1.3cm,thick,main node/.style={circle,draw,font=\Large,scale=0.4}]
  \node at (0,0.2) [main node] (A){$\beta_{1}$};
  \node at (1,0.2) [main node] (B){$\beta_{2}$};
  \node at (-0.2,0.8) (p1) {1};
  \node at (1.2,0.8) (p2) {2};
  \draw [-] (A) to (p1);
  \draw [-] (A) to (B);
  \draw [-] (B) to (p2); 
  \end{tikzpicture} \Bigg]
  = 0\,,
\end{align*}

$\bullet$ Coefficient of $a_{1}^{0}$: 
\begin{align*}
-\eta + b^{2}\psi_{2} + 2b\xi_{2}
  -
 b_{2}^{2} \Bigg[ \begin{tikzpicture}[baseline=3pt,->,>=,node distance=1.3cm,thick,main node/.style={circle,draw,font=\Large,scale=0.4}]
  \node at (0,0.2) [main node] (A){$\beta_{1}$};
  \node at (1,0.2) [main node] (B){$\beta_{2}$};
  \node at (-0.2,0.8) (p1) {1};
  \node at (0.2,0.8) (p2) {2};
  \draw [-] (A) to (p1);
  \draw [-] (A) to (B);
  \draw [-] (A) to (p2); 
  \end{tikzpicture} \Bigg]
  -
 b_{1}^{2} \Bigg[ \begin{tikzpicture}[baseline=3pt,->,>=,node distance=1.3cm,thick,main node/.style={circle,draw,font=\Large,scale=0.4}]
  \node at (0,0.2) [main node] (A){$\beta_{1}$};
  \node at (1,0.2) [main node] (B){$\beta_{2}$};
  \node at (-0.2,0.8) (p1) {1};
  \node at (1.2,0.8) (p2) {2};
  \draw [-] (A) to (p1);
  \draw [-] (A) to (B);
  \draw [-] (B) to (p2); 
  \end{tikzpicture} \Bigg]
= 0\,.
\end{align*} 
For more markings, relations can be obtained by pull-back via the morphism $\overline{\mathcal{M}}_{0,n,\beta}(X) \to \overline{\mathcal{M}}_{0,2,\beta}(X)$. The first two relations recover Lee-Pandharipande relations \cite{LP}. The first relation is independent of the target $X$. In fact,
the relation 
\begin{align*}
\psi_{1}
+
\psi_{2}
-
\Bigg[ \begin{tikzpicture}[baseline=3pt,->,>=,node distance=1.3cm,thick,main node/.style={circle,draw,font=\Large,scale=0.4}]
  \node at (0,0.2) [scale=.3,draw,circle,fill=black] (A){};
  \node at (1,0.2) [scale=.3,draw,circle,fill=black] (B){};
  \node at (-0.2,0.8) (p1) {1};
  \node at (1.2,0.8) (p2) {2};
  \draw [-] (A) to (p1);
  \draw [-] (A) to (B);
  \draw [-] (B) to (p2); 
  \end{tikzpicture} \Bigg]
  = 0
\end{align*}
holds in $A^{1}(\mathfrak{M}_{0,2})$. The second relation is a consequence of the following relation
\begin{equation*}
\psi_{2} - \Bigg[ \begin{tikzpicture}[baseline=3pt,->,>=,node distance=1.3cm,thick,main node/.style={circle,draw,font=\Large,scale=0.4}]
  \node at (0,0.2) [scale=.3,draw,circle,fill=black] (A){};
  \node at (1,0.2) [scale=.3,draw,circle,fill=black] (B){};
  \node at (0.8,0.8) (p1) {1};
  \node at (1.2,0.8) (p2) {2};
  \draw [-] (B) to (p1);
  \draw [-] (A) to (B);
  \draw [-] (B) to (p2); 
  \end{tikzpicture} \Bigg] = 0    
\end{equation*}
in $A^{1}(\mathfrak{M}_{0,3})$. After pulling back the relation to $\overline{\mathcal{M}}_{0,3,\beta}(X)$, multiplying with $\xi_{3}$ and pushing-forward to $\overline{\mathcal{M}}_{0,2,\beta}$ by forgetting the third marked point, we get the second relation. The third relation is a consequence of the second relation. The pull-back of the second relation to $\overline{\mathcal{M}}_{0,3,\beta}(X)$ is
\begin{align*}
\xi_{1} - \xi_{2} -b\psi_{2}
+
b \Bigg[ \begin{tikzpicture}[baseline=3pt,->,>=,node distance=1.3cm,thick,main node/.style={circle,draw,font=\Large,scale=0.4}]
  \node at (0,0.2) [main node] (A){$0$};
  \node at (1,0.2) [main node] (B){$\beta$};
  \node at (-0.2,0.8) (p1) {2};
  \node at (0.2,0.8) (p2) {3};
  \node at (1.2,0.8) (p3) {1};
  \draw [-] (A) to (p1);
  \draw [-] (A) to (p2); 
  \draw [-] (A) to (B);
  \draw [-] (B) to (p3); 
  \end{tikzpicture} \Bigg]
 +
 \sum_{\beta_{2} \neq 0} b_{1} \Bigg[ \begin{tikzpicture}[baseline=3pt,->,>=,node distance=1.3cm,thick,main node/.style={circle,draw,font=\Large,scale=0.4}]
  \node at (0,0.2) [main node] (A){$\beta_{1}$};
  \node at (1,0.2) [main node] (B){$\beta_{2}$};
  \node at (-0.2,0.8) (p1) {1};
  \node at (0.2,0.8) (p2) {3};
  \node at (1.2,0.8) (p3) {2};
  \draw [-] (A) to (p1);
  \draw [-] (A) to (p2);
  \draw [-] (A) to (B);
  \draw [-] (B) to (p3); 
  \end{tikzpicture} \Bigg]
 +
\sum_{\beta_{2} \neq 0} b_{1}^{2} \Bigg[ \begin{tikzpicture}[baseline=3pt,->,>=,node distance=1.3cm,thick,main node/.style={circle,draw,font=\Large,scale=0.4}]
  \node at (0,0.2) [main node] (A){$\beta_{1}$};
  \node at (1,0.2) [main node] (B){$\beta_{2}$};
  \node at (-0.2,0.8) (p1) {1};
  \node at (0.8,0.8) (p2) {2};
  \node at (1.2,0.8) (p3) {3};
  \draw [-] (A) to (p1);
  \draw [-] (A) to (B);
  \draw [-] (B) to (p2); 
  \draw [-] (B) to (p3);
  \end{tikzpicture} \Bigg]\,.
\end{align*} 
Take a cup product this relation with $\xi_{1}$ and push-forward to $\overline{\mathcal{M}}_{0,2,\beta}(X)$ which forgets the first marking. After relabelling the third marking to the first marking, the relation is exactly the third relation. 
\end{example}

\subsection{Genus 1 case}
In genus 1 cases, it is useful to replace $\psi$ classes with boundary strata. 
\begin{example}
 Let $g=1, n=1, d=2$. We substitute $a_{1} = b$ in $\mathsf{P}^{2}_{1,A,\beta}$:
 
$\bullet$ Coefficient of $m^{4}$:
\begin{align*}
0 = {} & (\eta^{2} - 4b^{2}\eta\psi_{1}  -4b\eta\xi_{1} + b^{4}\psi_{1}^{2} + 4b^{2}\xi_{1}^{2} + 4b^{3}\psi_{1}\xi_{1}) \\
& +
(2b_{2}^{2}\eta_{1} + 2b_{2}^{2}\eta_{2} - 2b^{2}b_{2}^{2}\psi_{1} - 4bb_{2}^{2}\xi_{1}) \Bigg[ \begin{tikzpicture}[baseline=3pt,->,>=,node distance=1.3cm,thick,main node/.style={circle,draw,font=\Large,scale=0.4}]
\node at (0,0.2) [main node] (A) {1, $\beta_{1}$};
\node at (1,0.2) [main node] (B) {$\beta_{2}$};
\node at (-0.2,0.8) (p1) {1};
\draw [-] (A) to (p1);
\draw [-] (A) to (B);
\end{tikzpicture} \Bigg] \\
& +
(2b_{1}^{2}\eta_{1} + 2b_{1}^{2}\eta_{2} - 2b^{2}b_{1}^{2}\psi_{1} - 4bb_{1}^{2}\xi_{1}) \Bigg[ \begin{tikzpicture}[baseline=3pt,->,>=,node distance=1.3cm,thick,main node/.style={circle,draw,font=\Large,scale=0.4}]
\node at (0,0.2) [main node] (A) {1, $\beta_{1}$};
\node at (1,0.2) [main node] (B) {$\beta_{2}$};
\node at (1.2,0.8) (p1) {1};
\draw [-] (B) to (p1);
\draw [-] (A) to (B);
\end{tikzpicture} \Bigg]\\
& + 
2(b-b_{1})^{2}b_{3}^{2} \Bigg[ \begin{tikzpicture}[baseline=3pt,->,>=,node distance=1.3cm,thick,main node/.style={circle,draw,font=\Large,scale=0.4}]
\node at (0,0.2) [main node] (A) {1, $\beta_{1}$};
\node at (0.8,0.2) [main node] (B) {$\beta_{2}$};
\node at (1.6,0.2) [main node] (C) {$\beta_{3}$};
\node at (-0.2,0.8) (p1) {1};
\draw [-] (A) to (p1);
\draw [-] (A) to (B);
\draw [-] (B) to (C);
\end{tikzpicture} \Bigg]
+
2b_{1}^{2}b_{3}^{2} \Bigg[ \begin{tikzpicture}[baseline=3pt,->,>=,node distance=1.3cm,thick,main node/.style={circle,draw,font=\Large,scale=0.4}]
\node at (0,0.2) [main node] (A) {1, $\beta_{1}$};
\node at (0.8,0.2) [main node] (B) {$\beta_{2}$};
\node at (1.6,0.2) [main node] (C) {$\beta_{3}$};
\node at (0.8,0.8) (p1) {1};
\draw [-] (B) to (p1);
\draw [-] (A) to (B);
\draw [-] (B) to (C);
\end{tikzpicture} \Bigg] 
+
2b_{1}^{2}(b-b_{3})^{2} \Bigg[ \begin{tikzpicture}[baseline=3pt,->,>=,node distance=1.3cm,thick,main node/.style={circle,draw,font=\Large,scale=0.4}]
\node at (0,0.2) [main node] (A) {1, $\beta_{1}$};
\node at (0.8,0.2) [main node] (B) {$\beta_{2}$};
\node at (1.6,0.2) [main node] (C) {$\beta_{3}$};
\node at (2,0.8) (p1) {1};
\draw [-] (C) to (p1);
\draw [-] (A) to (B);
\draw [-] (B) to (C);
\end{tikzpicture} \Bigg] \\
& + 
2b_{3}^{2}(b-b_{2})^{2} \Bigg[ \begin{tikzpicture}[baseline=3pt,->,>=,node distance=1.3cm,thick,main node/.style={circle,draw,font=\Large,scale=0.4}]
\node at (0,0.2) [main node] (A) {$\beta_{1}$};
\node at (1,0.2) [main node] (B) {1,$\beta_{2}$};
\node at (2,0.2) [main node] (C) {$\beta_{3}$};
\node at (-0.2,0.8) (p1) {1};
\draw [-] (A) to (p1);
\draw [-] (A) to (B);
\draw [-] (B) to (C);
\end{tikzpicture} \Bigg]
+
2b_{2}^{2}b_{3}^{2} \Bigg[ \begin{tikzpicture}[baseline=3pt,->,>=,node distance=1.3cm,thick,main node/.style={circle,draw,font=\Large,scale=0.4}]
\node at (0,0.2) [main node] (A) {$\beta_{1}$};
\node at (1,0.2) [main node] (B) {1,$\beta_{2}$};
\node at (2,0.2) [main node] (C) {$\beta_{3}$};
\node at (1,0.8) (p1) {1};
\draw [-] (B) to (p1);
\draw [-] (A) to (B);
\draw [-] (B) to (C);
\end{tikzpicture} \Bigg]\,,
\end{align*}
 
$\bullet$ Coefficient of $m^{2}$:
\begin{align*}
(-\eta + 2b\xi_{1}) \Bigg[ \begin{tikzpicture}[baseline=3pt,->,>=,node distance=1.3cm,thick,main node/.style={circle,draw,font=\Large,scale=0.4}]
\node at (0,0.2) [main node] (A) {$\beta$};
\node at (0,0.8) (p1) {1}; 
\draw [-] (p1) to (A);
\draw [-] (A) to [out=-45, in=-135,looseness=5] (A);
\end{tikzpicture}\Bigg]
-
b_{2}^{2} \Bigg[ \begin{tikzpicture}[baseline=3pt,->,>=,node distance=1.3cm,thick,main node/.style={circle,draw,font=\Large,scale=0.4}]
\node at (0,0.2) [main node] (A) {$\beta_{1}$};
\node at (1,0.2) [main node] (B) {$\beta_{1}$};
\node at (0,0.8) (p1) {1}; 
\draw [-] (A) to (B);
\draw [-] (p1) to (A);
\draw [-] (A) to [out=-45, in=-135,looseness=5] (A);
\end{tikzpicture}\Bigg]
+
(b^{2}-b_{1}^{2})\Bigg[ \begin{tikzpicture}[baseline=3pt,->,>=,node distance=1.3cm,thick,main node/.style={circle,draw,font=\Large,scale=0.4}]
\node at (0,0.2) [main node] (A) {$\beta_{1}$};
\node at (1,0.2) [main node] (B) {$\beta_{1}$};
\node at (1,0.8) (p1) {1}; 
\draw [-] (A) to (B);
\draw [-] (p1) to (B);
\draw [-] (A) to [out=-45, in=-135,looseness=5] (A);
\end{tikzpicture}\Bigg]
+
2b_{2}^{2}\Bigg[ \begin{tikzpicture}[baseline=3pt,->,>=,node distance=1.3cm,thick,main node/.style={circle,draw,font=\Large,scale=0.4}]
\node at (0,0.2) [main node] (A) {$\beta_{1}$};
\node at (1,0.2) [main node] (B) {$\beta_{2}$};
\node at (0,0.8) (p1) {1};
\draw [-] (A) to (p1);
\draw [-] (A) to [out=40, in=140] (B);
\draw [-] (A) to [out=-40, in=-140] (B);
\end{tikzpicture} \Bigg]
 = 0\,,
\end{align*}

$\bullet$ Coefficient of $m^{0}$: 
\begin{align*}
(\psi_{h} + \psi_{h'}) \Bigg[ \begin{tikzpicture}[baseline=3pt,->,>=,node distance=1.3cm,thick,main node/.style={circle,draw,font=\Large,scale=0.4}]
\node at (0,0.2) [main node] (A) {$\beta$};
\node at (0,0.8) (p1) {1}; 
\draw [-] (p1) to (A);
\draw [-] (A) to [out=-45, in=-135,looseness=5] (A);
\end{tikzpicture} \Bigg]
-
2\Bigg[ \begin{tikzpicture}[baseline=3pt,->,>=,node distance=1.3cm,thick,main node/.style={circle,draw,font=\Large,scale=0.4}]
\node at (0,0.2) [main node] (A) {$\beta_{1}$};
\node at (1,0.2) [main node] (B) {$\beta_{2}$};
\node at (0,0.8) (p1) {1}; 
\draw [-] (p1) to (A);
\draw [-] (A) to [out=40, in=140] (B);
\draw [-] (A) to [out=-40, in=-140] (B);
\end{tikzpicture} \Bigg]
= 
0\,.
\end{align*}
We could not show whether the first relation can be obtained from moduli spaces of curves.
On the other hand, the second  and the third relation comes from  tautological relations on moduli spaces of curves. 
\end{example}

\begin{example}
 Let $g=1, n=2, d=2$. There are 26 prestable graphs with at most two edges. The degree of $m$ can be either $0, 2, \textup{or} \, 4$. We prove that the relation which does not contain $\eta$ classes can be obtained from tautological relations on genus 0 prestable curves. The coefficient of $m^{4}$ and $a_{1}^{3}$ is the most complicated example. For simplicity, let $\delta$ be the codimension one boundary stratum of $\overline{\mathcal{M}}_{1,2,\beta}(X)$ associated to the one loop $X$-valued stable graph. After simplification, the relation becomes
\begin{align*}
0 = {} & b(\psi_{1}^{2} - \psi_{2}^{2}) + 2(\psi_{1}+\psi_{2})(\xi_{1} - \xi_{2}) \\ 
& +
(-2b_{2}\psi_{1}+2b_{1}\psi_{2} - 2\xi_{1} + 2\xi_{2} + (b_{1}-b_{2})(\psi_{h} + \psi_{h'})) \Bigg[ \begin{tikzpicture}[baseline=3pt,->,>=,node distance=1.3cm,thick,main node/.style={circle,draw,font=\Large,scale=0.4}]
\node at (0,0.2) [main node] (A) {1, $\beta_{1}$};
\node at (1,0.2) [main node] (B) {$\beta_{2}$};
\node at (-0.2,0.8) (p1) {1};
\node at (1.2,0.8) (p2) {2};
\draw [-] (A) to (p1);
\draw [-] (B) to (p2); 
\draw [-] (A) to (B);
\end{tikzpicture} \Bigg]
\\
& +
(-2b_{1}\psi_{1}+2b_{2}\psi_{2} - 2\xi_{1} + 2\xi_{2} + (b_{2} - b_{1})(\psi_{h} + \psi_{h'})) \Bigg[ \begin{tikzpicture}[baseline=3pt,->,>=,node distance=1.3cm,thick,main node/.style={circle,draw,font=\Large,scale=0.4}]
\node at (0,0.2) [main node] (A) {1, $\beta_{1}$};
\node at (1,0.2) [main node] (B) {$\beta_{2}$};
\node at (-0.2,0.8) (p1) {2};
\node at (1.2,0.8) (p2) {1};
\draw [-] (A) to (p1);
\draw [-] (B) to (p2); 
\draw [-] (A) to (B);
\end{tikzpicture} \Bigg]
\\
& +
(-2b_{1}+2b_{3}) \Bigg[ \begin{tikzpicture}[baseline=3pt,->,>=,node distance=1.3cm,thick,main node/.style={circle,draw,font=\Large,scale=0.4}]
\node at (0,0.2) [main node] (A) {1,$\beta_{1}$};
\node at (1,0.2) [main node] (B) {$\beta_{2}$};
\node at (2,0.2) [main node] (C) {$\beta_{3}$};
\node at (0,0.8) (p1) {1};
\node at (2,0.8) (p2) {2};
\draw [-] (A) to (B);
\draw [-] (B) to (C);
\draw [-] (A) to (p1);
\draw [-] (C) to (p2);
\end{tikzpicture} \Bigg]
+
(2b_{1}-2b_{3}) \Bigg[ \begin{tikzpicture}[baseline=3pt,->,>=,node distance=1.3cm,thick,main node/.style={circle,draw,font=\Large,scale=0.4}]
\node at (0,0.2) [main node] (A) {1,$\beta_{1}$};
\node at (1,0.2) [main node] (B) {$\beta_{2}$};
\node at (2,0.2) [main node] (C) {$\beta_{3}$};
\node at (0,0.8) (p1) {2};
\node at (2,0.8) (p2) {1};
\draw [-] (A) to (B);
\draw [-] (B) to (C);
\draw [-] (A) to (p1);
\draw [-] (C) to (p2);
\end{tikzpicture} \Bigg]
\\
& + (-2b_{1}+2b_{3}) \Bigg[ \begin{tikzpicture}[baseline=3pt,->,>=,node distance=1.3cm,thick,main node/.style={circle,draw,font=\Large,scale=0.4}]
\node at (0,0.2) [main node] (A) {$\beta_{1}$};
\node at (1,0.2) [main node] (B) {1,$\beta_{2}$};
\node at (2,0.2) [main node] (C) {$\beta_{3}$};
\node at (0,0.8) (p1) {1};
\node at (2,0.8) (p2) {2};
\draw [-] (A) to (p1);
\draw [-] (C) to (p2);
\draw [-] (A) to (B);
\draw [-] (B) to (C);
\end{tikzpicture} \Bigg]\,.
\end{align*}
\noindent
On $\overline{\mathcal{M}}_{1,n,\beta}(X)$, we have  
\begin{equation}
\psi_{1} = \frac{1}{12}\delta + \Bigg[ \begin{tikzpicture}[baseline=3pt,->,>=,node distance=1.3cm,thick,main node/.style={circle,draw,font=\Large,scale=0.4}]
\node at (0,0.2) [main node] (A) {1, $\beta_{1}$};
\node at (1,0.2) [main node] (B) {$\beta_{2}$};
\node at (0.8,0.8) (p1) {1};
\node at (1.2,0.8) (p2) {2};
\draw [-] (B) to (p1);
\draw [-] (B) to (p2); 
\draw [-] (A) to (B);
\end{tikzpicture} \Bigg]
+ 
\Bigg[ \begin{tikzpicture}[baseline=3pt,->,>=,node distance=1.3cm,thick,main node/.style={circle,draw,font=\Large,scale=0.4}]
\node at (0,0.2) [main node] (A) {1, $\beta_{1}$};
\node at (1,0.2) [main node] (B) {$\beta_{2}$};
\node at (-0.2,0.8) (p1) {2};
\node at (1.2,0.8) (p2) {1};
\draw [-] (A) to (p1);
\draw [-] (B) to (p2); 
\draw [-] (A) to (B);
\end{tikzpicture} \Bigg]\,,
\end{equation}
and similarly for $\psi_{2}$. The class $b(\psi_{1}^{2}-\psi_{2}^{2})$ is equal to
\begin{equation*}
b\Bigg[ \begin{tikzpicture}[baseline=3pt,->,>=,node distance=1.3cm,thick,main node/.style={circle,draw,font=\Large,scale=0.4}]
\node at (0,0.2) [main node] (A) {1, $\beta_{1}$};
\node at (1,0.2) [main node] (B) {$\beta_{2}$};
\node at (-0.2,0.8) (p1) {2};
\node at (1.2,0.8) (p2) {1};
\draw [-] (A) to (p1);
\draw [-] (B) to (p2); 
\draw [-] (A) to (B);
\end{tikzpicture} \Bigg]^{2}
-
\Bigg[ \begin{tikzpicture}[baseline=3pt,->,>=,node distance=1.3cm,thick,main node/.style={circle,draw,font=\Large,scale=0.4}]
\node at (0,0.2) [main node] (A) {1, $\beta_{1}$};
\node at (1,0.2) [main node] (B) {$\beta_{2}$};
\node at (-0.2,0.8) (p1) {1};
\node at (1.2,0.8) (p2) {2};
\draw [-] (A) to (p1);
\draw [-] (B) to (p2); 
\draw [-] (A) to (B);
\end{tikzpicture} \Bigg]^{2}
+ \frac{1}{6} \delta \Bigg[ \begin{tikzpicture}[baseline=3pt,->,>=,node distance=1.3cm,thick,main node/.style={circle,draw,font=\Large,scale=0.4}]
\node at (0,0.2) [main node] (A) {1, $\beta_{1}$};
\node at (1,0.2) [main node] (B) {$\beta_{2}$};
\node at (-0.2,0.8) (p1) {2};
\node at (1.2,0.8) (p2) {1};
\draw [-] (A) to (p1);
\draw [-] (B) to (p2); 
\draw [-] (A) to (B);
\end{tikzpicture} \Bigg]
- \frac{1}{6} \delta \Bigg[ \begin{tikzpicture}[baseline=3pt,->,>=,node distance=1.3cm,thick,main node/.style={circle,draw,font=\Large,scale=0.4}]
\node at (0,0.2) [main node] (A) {1, $\beta_{1}$};
\node at (1,0.2) [main node] (B) {$\beta_{2}$};
\node at (-0.2,0.8) (p1) {1};
\node at (1.2,0.8) (p2) {2};
\draw [-] (A) to (p1);
\draw [-] (B) to (p2); 
\draw [-] (A) to (B);
\end{tikzpicture} \Bigg]\,.
\end{equation*}
The excess intersection formula gives
\begin{equation*}
\Bigg[ \begin{tikzpicture}[baseline=3pt,->,>=,node distance=1.3cm,thick,main node/.style={circle,draw,font=\Large,scale=0.4}]
\node at (0,0.2) [main node] (A) {1, $\beta_{1}$};
\node at (1,0.2) [main node] (B) {$\beta_{2}$};
\node at (-0.2,0.8) (p1) {1};
\node at (1.2,0.8) (p2) {2};
\draw [-] (A) to (p1);
\draw [-] (B) to (p2); 
\draw [-] (A) to (B);
\end{tikzpicture} \Bigg]^{2}
= 
-(\psi_{h} + \psi_{h'}) \Bigg[ \begin{tikzpicture}[baseline=3pt,->,>=,node distance=1.3cm,thick,main node/.style={circle,draw,font=\Large,scale=0.4}]
\node at (0,0.2) [main node] (A) {1, $\beta_{1}$};
\node at (1,0.2) [main node] (B) {$\beta_{2}$};
\node at (-0.2,0.8) (p1) {1};
\node at (1.2,0.8) (p2) {2};
\draw [-] (A) to (p1);
\draw [-] (B) to (p2); 
\draw [-] (A) to (B);
\end{tikzpicture} \Bigg]
+
2\Bigg[ \begin{tikzpicture}[baseline=3pt,->,>=,node distance=1.3cm,thick,main node/.style={circle,draw,font=\Large,scale=0.4}]
\node at (0,0.2) [main node] (A) {1, $\beta_{1}$};
\node at (1,0.2) [main node] (B) {$\beta_{2}$};
\node at (2,0.2) [main node] (C) {$\beta_{3}$};
\node at (-0.2,0.8) (p1) {1};
\node at (2.2,0.8) (p2) {2};
\draw [-] (A) to (p1);
\draw [-] (C) to (p2); 
\draw [-] (A) to (B);
\draw [-] (B) to (C);
\end{tikzpicture} \Bigg]\,.
\end{equation*}
Using the $g = 0$, $n=2$ relation
\begin{align*}
\psi_{1}
+
\psi_{2}
-
\Bigg[ \begin{tikzpicture}[baseline=3pt,->,>=,node distance=1.3cm,thick,main node/.style={circle,draw,font=\Large,scale=0.4}]
  \node at (0,0.2) [main node] (A){$\beta_{1}$};
  \node at (1,0.2) [main node] (B){$\beta_{2}$};
  \node at (-0.2,0.8) (p1) {1};
  \node at (1.2,0.8) (p2) {2};
  \draw [-] (A) to (p1);
  \draw [-] (A) to (B);
  \draw [-] (B) to (p2); 
  \end{tikzpicture} \Bigg]
  = 0\,,
\end{align*}
the equation is equivalent to 
\begin{align*}
4(\xi_{1} - \xi_{2}) \Bigg(\frac{1}{12} \delta
+
\Bigg[ \begin{tikzpicture}[baseline=3pt,->,>=,node distance=1.3cm,thick,main node/.style={circle,draw,font=\Large,scale=0.4}]
\node at (0,0.2) [main node] (A) {1, $\beta_{1}$};
\node at (1,0.2) [main node] (B) {$\beta_{2}$};
\node at (0.8,0.8) (p1) {1};
\node at (1.2,0.8) (p2) {2};
\draw [-] (B) to (p1);
\draw [-] (B) to (p2); 
\draw [-] (A) to (B);
\end{tikzpicture} \Bigg]
\Bigg) 
+
\frac{b_{2}}{3} \Bigg[ \begin{tikzpicture}[baseline=3pt,->,>=,node distance=1.3cm,thick,main node/.style={circle,draw,font=\Large,scale=0.4}]
\node at (0,0.2) [main node] (A) {$\beta_{1}$};
\node at (1,0.2) [main node] (B) {$\beta_{2}$};
\node at (-0.2,0.8) (p1) {2};
\node at (1.2,0.8) (p2) {1};
\draw [-] (A) to (p1);
\draw [-] (B) to (p2); 
\draw [-] (A) to (B);
\draw [-] (A) to [out=-45, in=-135,looseness=5] (A);
\end{tikzpicture} \Bigg]
-
\frac{b_{2}}{3} \Bigg[ \begin{tikzpicture}[baseline=3pt,->,>=,node distance=1.3cm,thick,main node/.style={circle,draw,font=\Large,scale=0.4}]
\node at (0,0.2) [main node] (A) {$\beta_{1}$};
\node at (1,0.2) [main node] (B) {$\beta_{2}$};
\node at (-0.2,0.8) (p1) {1};
\node at (1.2,0.8) (p2) {2};
\draw [-] (A) to (p1);
\draw [-] (B) to (p2); 
\draw [-] (A) to (B);
\draw [-] (A) to [out=-45, in=-135,looseness=5] (A);
\end{tikzpicture} \Bigg]
= 0 \,.
\end{align*}
From the $g=0$ DR relation, $\xi_{1} - \xi_{2}$ can be written in terms of boundary strata
\begin{align*}
\xi_{1}-\xi_{2} = b\psi_{2} - b_{1} D(1|2)   
\end{align*}
where $D(1|2)$ is the sum over all codimension one boundary strata which splits the first and the second marked points. Using the relation
\begin{equation*}
\Bigg[ \begin{tikzpicture}[baseline=3pt,->,>=,node distance=1.3cm,thick,main node/.style={circle,draw,font=\Large,scale=0.4}]
\node at (0,0.2) [main node] (A){$\beta_{1}$};
\node at (1,0.2) [main node] (B){$\beta_{2}$};
\node at (-0.2,0.8) (p1) {1};  
\node at (0.2,0.8) (p2) {2};
\node at (1.2,0.8) (p3) {3};
\draw [-] (A) to (p1);
\draw [-] (A) to (p2);
\draw [-] (B) to (p3);
\draw [-] (A) to (B);
\end{tikzpicture} \Bigg]    
=
\Bigg[ \begin{tikzpicture}[baseline=3pt,->,>=,node distance=1.3cm,thick,main node/.style={circle,draw,font=\Large,scale=0.4}]
\node at (0,0.2) [main node] (A){$\beta_{1}$};
\node at (1,0.2) [main node] (B){$\beta_{2}$};
\node at (-0.2,0.8) (p1) {1};  
\node at (0.2,0.8) (p2) {3};
\node at (1.2,0.8) (p3) {2};
\draw [-] (A) to (p1);
\draw [-] (A) to (p2);
\draw [-] (B) to (p3);
\draw [-] (A) to (B);
\end{tikzpicture} \Bigg]
=
\Bigg[ \begin{tikzpicture}[baseline=3pt,->,>=,node distance=1.3cm,thick,main node/.style={circle,draw,font=\Large,scale=0.4}]
\node at (0,0.2) [main node] (A){$\beta_{1}$};
\node at (1,0.2) [main node] (B){$\beta_{2}$};
\node at (-0.2,0.8) (p1) {2};  
\node at (0.2,0.8) (p2) {3};
\node at (1.2,0.8) (p3) {1};
\draw [-] (A) to (p1);
\draw [-] (A) to (p2);
\draw [-] (B) to (p3);
\draw [-] (A) to (B);
\end{tikzpicture} \Bigg]
\end{equation*}
in $(g,n) = (0,3)$, the left hand side is equal to zero. This computation shows that the original relation is obtained from curves.
\end{example}
 
\subsection{Further directions} From the previous examples, we see that tautological relations are much richer than the relations pulled-back from tautological relations on the moduli spaces of stable curves. Since Theorem 3.3 produces relations uniformly for all target $X$, it is likely that, for a given $X$, further relations could also hold. The following questions are some possible directions for further studies.
\begin{enumerate}
\item {\em Lower genus cases}.  
From lower degree computations, Oprea conjectured that all tautological relations on $\Mbar_{0,n,\beta}(X)$ come from tautological relations on the moduli space of genus 0 curves \cite{O}. However, tautological relations on $\mathfrak{M}_{0,n}$ have not been investigated much yet and we do not know if there is a finite number of relations which replace the role of the WDVV equation in $\overline{\mathcal{M}}_{0,n}$.  In genus $1$ cases, tautological relations are generated by the WDVV relation and the Getzler's relation on $\overline{\mathcal{M}}_{1,4}$ \cite{Get}. When the genus is equal to $0$ or $1$, it is unknown whether there exists a finite number of relations which generate all tautological relations on the moduli space of stable maps to $X$.
\item {\em Generalization of the Pixton's 3-spin relations}. 
In \cite{P11, PP}, a system of tautological relations on $\overline{\mathcal{M}}_{g,n}$ was obtained from the study of  Witten's 3-spin classes. In \cite{J1}, the study of the equivariant GW theory of $\mathbb{P}^{1}$ also produces the 3-spin relations. The argument can be applied to produce tautological relations for a target space $X$ as follows. Consider a split vector bundle $V$ over $X$ and its projectivization $\mathbb{P}(V)$. The torus localization formula of the equivariant virtual fundamental class of $\Mbar_{g,n,\beta}(\mathbb{P}(V))$ relative to the moduli space of stable maps to $X$ was studied in \cite{Br,CG}. Then the pole cancellation technique applied to a variant of Givental's formalism \cite{Giv,J1} gives variants of Pixton's 3-spin relations on $\overline{\mathcal{M}}_{g,n,\beta}(X)$ twisted by Chern characters of $V$.
\end{enumerate}

\bibliographystyle{amsplain}

\end{document}